\documentclass[12pt]{amsart}

\newcommand{\hidden}[1]{
}

\usepackage{framed}
\usepackage{braket}
\usepackage{amsmath,amssymb,amsthm}
\usepackage{color}
\usepackage{bm}
\usepackage[all]{xy}
\usepackage{amscd}
\usepackage{graphicx}
\usepackage{ascmac}
\usepackage{mathrsfs}
\usepackage[top=30truemm,bottom=30truemm,left=25truemm,right=25truemm]{geometry}
\usepackage{mathtools}
\mathtoolsset{showonlyrefs=true}

\makeatletter
\@addtoreset{equation}{section}

\makeatother

\newcommand{\Z}{\mathbb{Z}}
\newcommand{\Q}{\mathbb{Q}}

\newcommand{\pe}{\mathfrak{p}}

\newcommand{\qu}{\mathfrak{q}}

\newcommand{\ttilde}{\widetilde}
\newcommand{\dual}{\vee}

\newcommand{\SF}[2]{\Fitt_{#2}^{\langle #1 \rangle}}
\newcommand{\vSF}[2]{\Fitt_{#2}^{[#1]}}

\newcommand{\PP}{\mathcal{P}}
\newcommand{\QQ}{\mathcal{Q}}
\newcommand{\CC}{\mathcal{C}}
\newcommand{\RR}{\mathcal{R}}
\newcommand{\GG}{\mathcal{G}}
\newcommand{\OO}{\mathcal{O}}
\newcommand{\EE}{\mathcal{E}}
\newcommand{\TT}{\mathcal{T}}

\newcommand{\FF}{\mathcal{F}}
\newcommand{\LL}{\mathcal{L}}
\newcommand{\bT}{\mathbb{T}}
\newcommand{\ff}{\mathfrak{f}}

\newcommand{\sI}{\mathscr{I}}

\newcommand{\iL}{\mathfrak{L}}

\newcommand{\otimesL}{\otimes^{\mathbb{L}}}

\newcommand{\De}{D}
\newcommand{\DeB}{\De^{\bdd}}
\newcommand{\DeP}{\De^{\perf}}
\newcommand{\DePTor}{\De^{\perf}_{\tor}}
\newcommand{\Ch}{Ch}
\newcommand{\ChB}{\Ch^{\bdd}}
\newcommand{\ChP}{\Ch^{\perf}}
\newcommand{\ChPTor}{\Ch^{\perf}_{\tor}}

\newcommand{\Inv}{\mathcal{I}}

\newcommand{\derR}{\mathcal{R}}

\newcommand{\ol}[1]{\overline{#1}}

\newcommand{\Zpur}{\Z_p^{\ur}}
\newcommand{\RRur}{\RR^{\ur}}

\DeclareMathOperator{\Gal}{Gal}

\DeclareMathOperator{\Ker}{Ker}

\DeclareMathOperator{\Image}{Im}

\DeclareMathOperator{\bdd}{b}
\DeclareMathOperator{\perf}{perf}
\DeclareMathOperator{\pd}{pd}
\DeclareMathOperator{\Fitt}{Fitt}
\DeclareMathOperator{\Det}{Det}
\DeclareMathOperator{\rank}{rank}
\DeclareMathOperator{\Hom}{Hom}
\DeclareMathOperator{\Frac}{Frac}
\DeclareMathOperator{\length}{d}

\DeclareMathOperator{\glob}{gl}
\DeclareMathOperator{\local}{loc}
\DeclareMathOperator{\Sel}{Sel}
\DeclareMathOperator{\Aut}{Aut}
\DeclareMathOperator{\ur}{ur}
\DeclareMathOperator{\tor}{tor}
\DeclareMathOperator{\ram}{ram}
\DeclareMathOperator{\Inf}{Inf}

\let\oldenumerate\enumerate
\renewcommand{\enumerate}{
   \oldenumerate
   \setlength{\itemsep}{1pt}
   \setlength{\parskip}{0pt}
   \setlength{\parsep}{0pt}
}
\let\olditemize\itemize
\renewcommand{\itemize}{
   \olditemize
   \setlength{\itemsep}{1pt}
   \setlength{\parskip}{0pt}
   \setlength{\parsep}{0pt}
}

\theoremstyle{plain}
\newtheorem{thm}{Theorem}[section]
\newtheorem{lem}[thm]{Lemma}

\newtheorem{prop}[thm]{Proposition}
\newtheorem{cor}[thm]{Corollary}

\theoremstyle{definition}
\newtheorem{defn}[thm]{Definition}
\newtheorem{rem}[thm]{Remark}

\title[Two variable equivariant Iwasawa theory]{Fitting ideals in two-variable equivariant Iwasawa theory and an application}
\author[T. Kataoka]{Takenori Kataoka}
\address{Faculty of Science and Technology, Keio University.
3-14-1 Hiyoshi, Kohoku-ku, Yokohama, Kanagawa 223-8522, Japan}
\email{tkataoka@math.keio.ac.jp}

\begin{document}

\begin{abstract}

We study equivariant Iwasawa theory for two-variable abelian extensions of an imaginary quadratic field.
One of the main goals of this paper is to describe the Fitting ideals of Iwasawa modules using $p$-adic $L$-functions.
We also provide an application to Selmer groups of elliptic curves with complex multiplication.

\end{abstract}

\maketitle


\section{Introduction}\label{sec:01}
Main conjectures in Iwasawa theory predict that, in various situations, Iwasawa modules are closely related to $p$-adic $L$-functions.
As a significant result, Wiles \cite{Wil90} proved the main conjecture for ideal class groups over totally real fields.
The work studied the cyclotomic $\Z_p$-extensions of totally real fields, and we call the setting ``one-variable.''
His work was refined by, among others, Ritter-Weiss from the viewpoint of equivariant theory.
In fact, Ritter-Weiss \cite{RW02} proved the equivariant main conjecture for finite abelian extensions of totally real fields.
Moreover, in subsequent works they succeeded in proving the equivariant main conjecture even for non-abelian Galois extensions, but we do not discuss non-abelian cases in this paper.

Another important theme in Iwasawa theory is the ``two-variable'' analogue, that is, study of the unique $\Z_p^2$-extensions of imaginary quadratic fields.
Using the Euler system of elliptic units, Rubin \cite{Rub91} proved the two-variable main conjecture.
As for equivariant refinements, Johnson-Leung--Kings \cite{JK11} proved a formulation of two-variable equivariant main conjecture.

However, in order to formulate the equivariant main conjectures in both one-variable and two-variable settings, we have to suitably modify the Iwasawa modules.
As a consequence, it is not clear how to recover precise information about the original Iwasawa modules.
There seems to be agreement that this kind of information is afforded by the Fitting ideals of the modules (we only study the initial Fitting ideals in this paper).

In the one-variable situation, Greither-Kurihara \cite{GK15} \cite{GK17} developed a method to compute the Fitting ideals of the original Iwasawa modules.
See Remark \ref{rem:109} for detail and further progress.
The main theme of this paper is to develop a two-variable analogue of those one-variable results.

To be more precise, we fix our notation.
Let $K$ be an imaginary quadratic field.
We fix a prime number $p \geq 5$ which splits in $K$ into $\pe$ and $\overline{\pe}$.
We fix an algebraic closure $\overline{K}$ of $K$ and every algebraic extension will be considered to be contained in $\overline{K}$.
Let $K_{\infty}/K$ be the unique $\Z_p^2$-extension.
Let $L/K$ be a finite abelian extension and put $L_{\infty} = K_{\infty}L$.
Put $\GG = \Gal(L_{\infty}/K)$, the Galois group, and $\RR = \Z_p[[\GG]]$, the Iwasawa algebra.

In general, for a finite extension $k$ of $\Q$, let $S_p(k)$ be the set of $p$-adic primes of $k$.
If $\FF/k$ is an algebraic extension, we denote by $S_{\ram}(\FF/k)$ the set of finite places of $k$ which are ramified in $\FF/k$.
For example, we have $S_{\ram}(L_{\infty}/K) = S_{\ram}(L/K) \cup S_p(K)$ and $S_p(K) = \{\pe, \ol{\pe}\}$.

\subsection{$\Sigma_0$-ramified Iwasawa module}

Let $\Sigma$ be a finite set of finite places of $K$ which contains $S_{\ram}(L_{\infty}/K)$.
Put $\Sigma_0 = \Sigma \setminus \{\overline{\pe}\}$.
As the first object of study, we consider the $\Sigma_0$-ramified Iwasawa module $X_{\Sigma_0}(L_{\infty})$, which is defined as the Galois group of the maximal abelian $p$-extension of $L_{\infty}$ which is unramified outside $\Sigma_0$.
It is known that $X_{\Sigma_0}(L_{\infty})$ is a finitely generated torsion $\RR$-module.

The first goal of this paper is to compute its initial Fitting ideal $\Fitt_{\RR} (X_{\Sigma_0}(L_{\infty}))$.
%
A significant progress toward this problem was made in a previous paper \cite{Kata_05} of the author.
Let $\Fitt_{\RR}^{[n]}$ denote the $n$-th shift of $\Fitt_{\RR}$, which was introduced in \cite{Kata_05}; see Theorem \ref{thm:81} for the definition.
Then it is shown in \cite[Theorem 5.16]{Kata_05} that the ideals $\Fitt_{\RR}(X_{\Sigma_0}(L_{\infty}))$ and $\Fitt_{\RR}^{[2]}(\Z_p)$ differ only by an invertible ideal.
Moreover, the author claimed that the invertible ideal should be described as a certain $p$-adic $L$-function.
Note here that invertible ideals of $\RR$ are principal since $\RR$ is a product of local rings.
 

Now we state the first main result in this paper, which asserts that the prediction is true.
In Definition \ref{defn:96}, we will introduce an invertible ideal $\iL_{\Sigma_0}(L_{\infty}/K)$ of $\RR$ which is related to a Katz-type $p$-adic $L$-function.

\begin{thm}\label{thm:33}
Let $\Sigma$ be a finite set of finite places of $K$ which contains $S_{\ram}(L_{\infty}/K)$ and put $\Sigma_0 = \Sigma \setminus \{\ol{\pe}\}$.
Then we have 
\[
\Fitt_{\RR}(X_{\Sigma_0}(L_{\infty})) = \iL_{\Sigma_0}(L_{\infty}/K) \Fitt_{\RR}^{[2]}(\Z_p)
\]
as ideals of $\RR$.
\end{thm}

The proof of Theorem \ref{thm:33} will be outlined just after Theorem \ref{thm:34} below.

\subsection{$S$-ramified Iwasawa module}

More generally than $X_{\Sigma_0}(L_{\infty})$, we consider the $S$-ramified Iwasawa module $X_{S}(L_{\infty})$ for any finite set $S$ of finite places of $K$ such that $\pe \in S$ and $\ol{\pe} \not \in S$.
In particular, we are concerned with the minimal case $S = \{\pe\}$.
Although $X_{\{\pe\}}(L_{\infty})$ looks more fundamental than $X_{\Sigma_0}(L_{\infty})$, the structure of $X_{\{\pe\}}(L_{\infty})$ is more complicated from our perspective.

\begin{rem}\label{rem:109}
Let us recall the analogue in the one-variable setting. 
We consider a finite abelian extension $k'/k$ of totally real fields and the cyclotomic $\Z_p$-extension $k'_{\infty}$ of $k'$.

Let $\Sigma$ be a finite set of finite places of $k$ which contains $S_{\ram}(k'_{\infty}/k) = S_{\ram}(k'/k) \cup S_p(k)$.
Then the Fitting ideal of $X_{\Sigma}(k'_{\infty})$ is described in \cite{GK15} and \cite{GK17}, complemented by \cite{GKT19}.
The author \cite{Kata_05} gave an interpretation of the results \cite{GK15}, \cite{GK17} using $\Fitt^{[2]}(\Z_p)$.
Our Theorem \ref{thm:33} is an analogue of those results.

However, those results do not describe the Fitting ideal of $X_{S_p(k)}(k'_{\infty})$ unless accidentally $S_{\ram}(k'_{\infty}/k) = S_p(k)$.
It is a more recent work \cite{GKK_09} that describes the Fitting ideal of $X_{S_p(k)}(k'_{\infty})$ in general, using $\Fitt^{[1]}$ of a bit more complicated modules.
Our second main result in this paper below is an analogue of this result.
\end{rem}

In our two-variable setting, we obtain the following result.

\begin{thm}\label{thm:34}
Let $S$ be a finite set of finite places of $K$ such that $\pe \in S$ and $\ol{\pe} \not \in S$.
Take a finite set $\Sigma$ of finite places of $K$ which contains $S_{\ram}(L_{\infty}/K) \cup S$.
Then we have
\[
\Fitt_{\RR}(X_{S}(L_{\infty})) = \iL_{\Sigma, S}(L_{\infty}/K) \Fitt_{\RR}^{[1]}(Z_{\Sigma \setminus S}^0(L_{\infty}/K))
\]
as ideals of $\RR$.
\end{thm}

The definitions of $Z_{\Sigma \setminus S}^0(L_{\infty}/K)$ and $\iL_{\Sigma, S}(L_{\infty}/K)$ will be given in Section \ref{sec:51}.
See also Remark \ref{rem:111} for the computation of $\Fitt_{\RR}^{[1]}(Z_{\Sigma \setminus S}^0(L_{\infty}/K))$.

By Lemma \ref{lem:79}(1) below and $\iL_{\Sigma_0} = \iL_{\Sigma, \Sigma_0}$ by Definition \ref{defn:96}, it follows that Theorem \ref{thm:34} for $S = \Sigma_0$ is equivalent to Theorem \ref{thm:33}.
For that reason, in this paper we give only a proof of Theorem \ref{thm:34}.

Now we outline the proof of Theorem \ref{thm:34}.
The main ingredient is the result by Johnson-Leung--Kings \cite{JK11}, which is mentioned above.
However, we have to develop algebraic theory to connect the result \cite[Theorem 5.7]{JK11}, which uses determinant modules, and Theorem \ref{thm:34}.
Indeed, in this paper we will show that the notions of {\it determinant modules of complexes} and {\it Fitting ideals of modules} are essentially equivalent.
Then we reduce the proof of Theorem \ref{thm:34} to Theorem \ref{thm:42} on the determinant of a certain complex $C_{\Sigma, S}$, and prove Theorem \ref{thm:42} via the result of \cite{JK11}.
Here we also need the work by de Shalit \cite{dS87}, since the main result of \cite{JK11} uses elliptic units instead of $p$-adic $L$-functions.

\subsection{Application to CM elliptic curves}\label{subsec:302}

It is known that the Iwasawa modules over imaginary quadratic fields are related to the Selmer groups of elliptic curves with complex multiplication.
That fact is used to prove results of Birch--Swinnerton-Dyer type (for example, Rubin \cite[\S 11, \S12]{Rub91}, de Shalit \cite[Chapter IV]{dS87}).
Motivated by those works, we apply the results of the preceding sections to elliptic curves with complex multiplication.

Let $E$ be an elliptic curve over $\Q$ which has complex multiplication by the ring of integers $\OO_K$ of $K$.
Suppose that $E$ has good reduction at $p$; since $p$ splits in $K$, the elliptic curve $E$ must have ordinary reduction at $p$.
Similarly as above, let $L/K$ be a finite abelian extension and $S$ a finite set of finite places of $K$ such that $\pe \in S$ and $\ol{\pe} \not\in S$.
We shall study the $\pe$-torsion part of the $S$-imprimitive Selmer group $\Sel_S(E/L)[\pe^{\infty}]$ (introduced in Definition \ref{defn:66}).

In order to state the result, put $L^{\pe} = L(E[\pe^{\infty}])$, which is a one-variable extension of $K$.
Put $\RR^{\pe} = \Z_p[[\Gal(L^{\pe}/K)]]$ and let $\pi_{L^{\pe}/L}: \RR^{\pe} \to \Z_p[\Gal(L/K)]$ be the projection map.
We denote by $\chi_{E, \pe}$ the character defined by the action of $E[\pe^{\infty}]$:
\[
\chi_{E, \pe}: \Gal(L^{\pe}/K) \to \Aut(E[\pe^{\infty}]) \simeq \Z_p^{\times}.
\]
Then $\chi_{E, \pe}$ induces a twisting algebra isomorphism $\RR^{\pe} \to \RR^{\pe}$, which we denote by $\ttilde{\chi_{E, \pe}}$.
We denote by $(-)^{\dual}$ the Pontryagin dual.

\begin{thm}\label{thm:75}
Let $\Sigma$ be a finite set of finite places of $K$ which contains $S_{\ram}(L_{\infty}/K) \cup S$.
Then we have
\[
\Fitt_{\Z_p[\Gal(L/K)]} (\Sel_{S}(E/L)[\pe^{\infty}]^{\dual}) 
= \pi_{L^{\pe}/L} \circ \ttilde{\chi_{E, \pe}} \left( \iL_{\Sigma, S}(L^{\pe}/K) \Fitt_{\RR^{\pe}}^{[1]}(Z_{\Sigma \setminus S}^0(L^{\pe}/K)) \right)
\]
as ideals of $\Z_p[\Gal(L/K)]$.
%
\end{thm}

The definitions of $Z_{\Sigma \setminus S}^0(L^{\pe}/K)$ and $\iL_{\Sigma, S}(L^{\pe}/K)$ will be given in Section \ref{sec:51}.
Note that the key ingredient of the proof is Theorem \ref{thm:42} rather than Theorem \ref{thm:34}.

\begin{rem}
The fractional ideal $\Fitt^{[1]}_{\RR^{\pe}} (Z_{\Sigma \setminus S}^0(L^{\pe}/K))$ is not contained in $\RR^{\pe}$ in general, while the product with $\iL_{\Sigma, S}(L^{\pe}/K)$ is contained in $\RR^{\pe}$.
In other (quite rough) words, $\Fitt^{[1]}_{\RR^{\pe}} (Z_{\Sigma \setminus S}^0(L^{\pe}/K))$ has poles, but there are canceled by the zeros of $\iL_{\Sigma, S}(L^{\pe}/K)$.
For that reason, the right hand side of the formula in Theorem \ref{thm:75} cannot be decomposed as the product of two ideals of $\Z_p[\Gal(L/K)]$.
\end{rem}

\subsection{Outline of this paper}\label{subsec:201}

The rest of this paper is organized as follows.
In Section \ref{sec:51}, in order to complete the statements of the main theorems, we give precise definitions of several objects like $\Fitt^{[n]}$.
In Section \ref{sec:02}, we show the general relation between determinant modules and Fitting ideals.
In Section \ref{sec:105}, we review properties of various arithmetic complexes in derived categories.
In Section \ref{sec:107}, we prove Theorem \ref{thm:34} in the way we already outlined.
In Section \ref{sec:106}, we deduce Theorem \ref{thm:75} from Theorem \ref{thm:34}.

\section{Definitions and preliminaries}\label{sec:51}

In this section, we give the definitions and basic properties of the ingredients in the statements of the main theorems in Section \ref{sec:01}.

\subsection{Shifts of $\Fitt$}\label{subsec:91}

In this subsection, we review the results of \cite{Kata_05} on shifts of Fitting invariants.

Though we can deal with more general commutative rings, for the sake of simplicity, we only consider the following situation.
Let $G$ be a profinite abelian group and consider the completed group ring $R = \Z_p[[G]]$.
Suppose that $G$ has an open subgroup $G'$ which is isomorphic to $\Z_p^d$ for some $d \geq 0$.
Put $R' = \Z_p[[G']]$, which is isomorphic to the ring of formal power series of $d$ variables over $\Z_p$.
Note that the choice of $G'$ is auxiliary and the following argument does not depend on $G'$.

An $R$-module is said to be torsion (resp. pseudo-null) if it is torsion (resp. pseudo-null) as an $R'$-module.
For an $R$-module $X$, we denote by $\pd_{R}(X)$ the projective dimension of $X$ over $R$.
As already mentioned, any invertible ideal of $R$ is principal.

\begin{thm}[{ \cite[Theorem 2.6]{Kata_05}}]\label{thm:81}
For each finitely generated torsion $R$-module $X$ and an integer $n \geq 0$, we define a fractional ideal $\vSF{n}{R}(X)$ as follows.
Take an exact sequence 
\[
0 \to Y \to P_1 \to \dots \to P_n \to X \to 0
\]
of finitely generated torsion $R$-modules such that $\pd_{R}(P_i) \leq 1$ for $1 \leq i \leq n$.
Then we define
\[
\vSF{n}{R}(X) = \left( \prod_{i=1}^n \Fitt_{R}(P_i)^{(-1)^i} \right) \Fitt_{R}(Y).
\]
Then this is well-defined, that is, independent from the choice of the resolution.
\end{thm}

We also use the following variant.

\begin{thm}[{\cite[Theorem 3.20, Corollary 3.21]{Kata_05}}] \label{thm:82}
For each finitely generated torsion $R$-module $X$ and every integer $n$, we have a fractional ideal $\SF{n}{R}(X)$ satisfying the following properties.

(1) 
If $X$ satisfies $\pd_{R'}(X) \leq 1$, then we have $\SF{0}{R}(X) = \Fitt_{R}(X)$.

(2)
If $\pd_{R}(X) < \infty$, then $\SF{n}{R}(X)$ is an invertible ideal for any $n$.
Moreover, we have $\SF{n}{R}(X) = \SF{0}{R}(X)^{(-1)^n}$.

(3)
Let $0 \to Y \to Q_1 \to \dots \to Q_d \to X \to 0$ be an exact sequence of finitely generated torsion $R$-modules.
Suppose $\pd_{R}(Q_i) < \infty$ for $1 \leq i \leq d$.
Then we have
\[
\SF{n}{R}(X) = \left( \prod_{i=1}^d \SF{n-d+i}{R}(Q_i) \right) \SF{n-d}{R}(Y)
\]
for any $n$.
\end{thm}

The role of $\SF{n}{\RR}$ in this paper is auxiliary compared to $\Fitt^{[n]}_{\RR}$.
An advantage is that $\SF{n}{\RR}$ behaves better with respect to exact sequences.
A disadvantage is that we do not have $\SF{0}{\RR}(X) = \Fitt_{\RR}(X)$ in general, while we always have $\Fitt^{[0]}_{\RR}(X) = \Fitt_{\RR}(X)$.
We have a sufficient condition for $\SF{n}{\RR}$ and $\Fitt^{[n]}_{\RR}$ to coincide:

\begin{lem}[{\cite[Remark 3.25]{Kata_05}}]\label{lem:83}
Let $n$ be non-negative.
Then the equality 
\[
\SF{n}{R}(X) = \vSF{n}{R}(X)
\]
holds if $\pd_{R'}(X) \leq n + 1$.
\end{lem}

We record a simple proposition.

\begin{prop}\label{prop:22}
Let $Q$ be a pseudo-null $R$-module such that $\pd_{R}(Q) < \infty$.
Then we have $\SF{0}{R}(Q) = R$.
\end{prop}

\begin{proof}
By $\pd_{R}(Q) < \infty$, there exists an exact sequence
\[
0 \to P_d \to \dots \to P_2 \to P_1 \to Q \to 0
\]
of finitely generated torsion $R$-modules with $\pd_{R}(P_i) \leq 1$.
For any prime ideal $\qu$ of $R$ of height 1, since $Q_{\qu} = 0$, we have an exact sequence
\[
0 \to (P_d)_{\qu} \to \dots \to (P_2)_{\qu} \to (P_1)_{\qu} \to 0.
\]
It follows that
\[
\left(\prod_{1 \leq i \leq d} \Fitt_{R}(P_i)^{(-1)^i}\right) R_{\qu}
= \prod_{1 \leq i \leq d} \Fitt_{R_{\qu}}((P_i)_{\qu})^{(-1)^i} 
= R_{\qu}.
\]
Then Lemma \ref{lem:23} below implies
\[
\prod_{1 \leq i \leq d} \Fitt_{R}(P_i)^{(-1)^i} = R.
\]
By the properties of $\SF{n}{\RR}$, the left hand side equals to $\SF{0}{R}(Q)$.
\end{proof}

\begin{lem}\label{lem:23}
Let $I, J$ be invertible ideals of $R$.
Suppose that $I R_{\qu} = J R_{\qu}$ holds for every height one prime ideal $\qu$.
Then we have $I = J$.
\end{lem}

\begin{proof}
This lemma is more or less well-known, but we give a proof for convenience.
Consider the natural injective map $I / (I \cap J) \hookrightarrow R/J$.
By assumption, the module $I / (I \cap J)$ is pseudo-null, while $R/J$ does not contain non-trivial pseudo-null submodules.
Hence $I \subset J$.
By symmetry, we conclude that $I = J$ holds.
\end{proof}

\subsection{Modules $Z^0$}

Let us return to the setting in Section \ref{sec:01}.
We define various $\RR$-modules in the same way as in \cite[Subsection 1.1]{GKK_09}.

\begin{defn}\label{defn:52}
Let $\FF$ be an abelian extension of $K$.
For each finite place $v$ of $K$, let $D_v(\FF/K)$ be the decomposition group of $\FF/K$ at $v$ and put
\[
Z_v(\FF/K) = \Z_p[[\Gal(\FF/K)/D_v(\FF/K)]].
\]
We regard $Z_v(\FF/K)$ as a (cyclic) $\Z_p[[\Gal(\FF/K)]]$-module.

For any finite set $A$ of finite places of $K$, 
put 
\[
Z_{A}(\FF/K) = \bigoplus_{v \in A} Z_v(\FF/K).
\]
When $A \neq \emptyset$, we define $Z_{A}^0(\FF/K)$ as the kernel of the augmentation map $Z_{A}(\FF/K) \to \Z_p$.
Then we have an exact sequence
\[
0 \to Z_{A}^0(\FF/K) \to Z_{A}(\FF/K) \to \Z_p \to 0.
\]
These are also regarded as $\Z_p[[\Gal(\FF/K)]]$-modules.
\end{defn}

\begin{rem}\label{rem:85}
Suppose $\FF = L_{\infty}$, which is a two-variable extension of $K$.

If $v$ is outside $p$, then $Z_v(L_{\infty}/K)$ is torsion but not pseudo-null as an $\RR$-module.

On the other hand, suppose $v$ is one of $\pe, \ol{\pe}$.
Then it is well-known that $v$ splits finitely in $L_{\infty}/K$, so $Z_v(L_{\infty}/K)$ is a pseudo-null $\RR$-module.
By the local class field theory, $D_v(L_{\infty}/K)$ is a quotient of the profinite completion of $K_v^{\times} \simeq \Q_p^{\times}$.
It follows that $D_v(L_{\infty}/K)$ is $p$-torsion-free (see also \cite[Lemma 5.14]{Kata_05}).
Therefore, we have $\pd_{\RR}(Z_v(L_{\infty}/K)) < \infty$.
In particular, Proposition \ref{prop:22} shows that 
\begin{equation}\label{eq:108}
\SF{0}{\RR}(Z_{v}(L_{\infty}/K)) = \RR.
\end{equation}
\end{rem}

\begin{rem}\label{rem:86}
Suppose $\FF = L^{\pe} = L(E[\pe^{\infty}])$, which is a one-variable extension of $K$, as in Subsection \ref{subsec:302}.
Then every finite place $v$ splits finitely in $L^{\pe}$, so $Z_v(L^{\pe}/K)$ is torsion but not pseudo-null as an $\RR^{\pe}$-module.
\end{rem}

Next we show a lemma on the $\Fitt^{[1]}$ in the right hand side of Theorem \ref{thm:34}.
The first assertion enables us to deduce Theorem \ref{thm:33} from Theorem \ref{thm:34}.

\begin{lem}\label{lem:79}
The following are true.

(1) In the case where $S = \Sigma_0$, we have
\[
 \Fitt^{[1]}_{\RR}(Z_{\{\pe\}}^0(L_{\infty}/K)) = \Fitt^{[2]}_{\RR}(\Z_p).
\]

(2) In the case where $S \subsetneqq \Sigma_0$, we have
\[
\Fitt^{[1]}_{\RR}(Z_{\Sigma \setminus S}^0(L_{\infty}/K))
= \Fitt^{[1]}_{\RR}(Z_{\Sigma_0 \setminus S}^0(L_{\infty}/K)).
\]
\end{lem}

\begin{proof}
(1)
We have an exact sequence
\[
0 \to Z_{\{\pe\}}^0(L_{\infty}/K) \to Z_{\{\pe\}}(L_{\infty}/K) \to \Z_p \to 0
\]
of $\RR$-modules.
Therefore, \eqref{eq:108} and the properties of $\SF{n}{\RR}$ show
\[
\SF{2}{\RR}(\Z_p) = \SF{1}{\RR}(Z_{\{\pe\}}^0(L_{\infty}/K)).
\]
By Lemma \ref{lem:83}, this equality says
\[
\Fitt^{[2]}_{\RR}(\Z_p) = \Fitt^{[1]}_{\RR}(Z_{\{\pe\}}^0(L_{\infty}/K)).
\]

(2)
We have an exact sequence
\[
0 \to Z_{\Sigma_0 \setminus S}^0(L_{\infty}/K) 
\to Z_{\Sigma \setminus S}^0(L_{\infty}/K)
\to Z_{\{\pe\}}(L_{\infty}/K)
\to 0.
\]
Similarly as in (1), the assertion follows from \eqref{eq:108} and the properties of $\SF{n}{\RR}$, $\Fitt^{[n]}_{\RR}$.
\end{proof}

\begin{rem}\label{rem:110}
Lemma \ref{lem:79}(2) does not play a practical role in this paper, but it clarifies the analogy with the one-variable setting in Remark \ref{rem:109}.
In the one-variable setting, the Fitting ideal of $X_{S_p(k)}(k'_{\infty})$ is described in \cite{GKK_09} by $\Fitt^{[1]}$ of $Z^0$ at {\it non $p$-adic primes}.
In our two-variable setting, Theorem \ref{thm:34} and Lemma \ref{lem:79}(2) show that the analogy holds for $X_{\{\pe\}}(L_{\infty})$.
\end{rem}

\begin{rem}\label{rem:111}
In \cite[Section 4.3]{Kata_05}, we illustrated how to compute $\Fitt^{[2]}_{\RR}(\Z_p)$ explicitly.
The computation is possible in principle, though that gets more complicated when the $p$-rank of $\Gal(L_{\infty}/K_{\infty})$ gets larger.

On the other hand, it seems quite hard to compute $\Fitt^{[1]}_{\RR}(Z_{\Sigma_0 \setminus S}^0(L_{\infty}/K))$ when $S \subsetneqq \Sigma_0$, because the decomposition fields of $v \in \Sigma_0 \setminus S$ can be diverse one-variable extensions of $K$ in $L_{\infty}$.
This is one feature in the two-variable setting; in the one-variable setting in Remark \ref{rem:109}, the decomposition fields are at any rate finite extensions of $k$.
\end{rem}

\subsection{$p$-adic $L$-functions}


As we will recall later (see the proof of Theorem \ref{thm:94}), the results of de Shalit \cite{dS87} involve a base change from $\Z_p$ to $\Zpur$.
Here, $\Zpur$ is the ring of integers of the completion $\Q_p^{\ur}$ of the maximal unramified extension of $\Q_p$.
We will freely use the following lemma that fractional ideals of Iwasawa algebras are characterized by their base changes to $\Zpur$.

\begin{lem}\label{lem:90}
Let $G$ be a profinite group as in Subsection \ref{subsec:91}.
Put $R = \Z_p[[G]]$ and $R^{\ur} = \Zpur[[G]]$.
Note that each fractional ideal $I$ of $R$ yields a fractional ideal $I R^{\ur}$ of $R^{\ur}$.
Then, for fractional ideals  $I$ and $J$ of $R$, we have $I R^{\ur} = J R^{\ur}$ if and only if $I = J$.
\end{lem}

\begin{proof}
This lemma is a special case of the theory of faithfully flat descent.
\end{proof}

%

Now we introduce $p$-adic $L$-functions.
See \cite[Theorem II.4.14]{dS87} for more information.

At first we consider the following specific situation associated to a nonzero integral ideal $\ff$ of $K$ which is prime to $p$.
Put $L_{\ff} = K(\ff p)$, the ray class field of $K$ modulo $\ff p$, so $L_{\ff, \infty} = K(\ff p^{\infty})$.
We take $\Sigma_{\ff}$ as the set of prime divisors of $\ff p$.
Put $\GG_{\ff} = \Gal(L_{\ff, \infty}/K)$, $\RR_{\ff} = \Z_p[[\GG_{\ff}]]$, and $\RRur_{\ff} = \Zpur[[\GG_{\ff}]]$.

\begin{thm}[{\cite[Theorem II.4.14, Corollary II.6.7]{dS87}}]\label{thm:92}
There exists an element $\mu_{\ff} \in \RRur_{\ff}$ (which is denoted by $\mu(\ff \overline{\pe}^{\infty})$ in \cite{dS87}) satisfying the following.
Let $\varepsilon$ be any grossencharacter of conductor dividing $\ff p^{\infty}$ and of infinity type $(k, j)$ with $k > 0$ and $j \leq 0$.
Then we have
\begin{equation}\label{eq:93}
\Omega_p^{j-k} \cdot \varepsilon(\mu_{\ff})
= \Omega^{j-k} \cdot \left(\frac{\sqrt{d_K}}{2 \pi} \right)^j 
\cdot  G(\varepsilon) \cdot  \left(1-\frac{\varepsilon(\pe)}{p} \right)
\widehat{L}_{\Sigma_{\ff} \setminus \{\pe\}} (\varepsilon^{-1}, 0),
\end{equation}
where $\Omega$ and $\Omega_p$ are complex and $p$-adic periods, $-d_K$ is the discriminant of $K$, $G(\varepsilon)$ is ``like Gauss sum'' (in the words of \cite{dS87}), and $\widehat{L}_{\Sigma_{\ff} \setminus \{\pe\}}$ is the completed $L$-function without Euler factors at the places in $\Sigma_{\ff} \setminus \{\pe\}$.
\end{thm}

The crucial property of $\mu_{\ff}$ is that it comes from elliptic units; see Theorem \ref{thm:94}.

Now, returning to the general situation, we introduce $p$-adic $L$-functions by similar formulas to $\mu_{\ff}$.
Let $\RRur = \Zpur[[\GG]]$ be the completed group ring of $\GG$ over $\Zpur$.
Note that, as a convention of this paper, $\iL$ denotes invertible ideals of Iwasawa algebras, while $\LL$ denotes elements.

\begin{defn}\label{defn:96}
(1) We define an element $\LL_{\Sigma, S}(L_{\infty}/K) \in \Frac(\RRur)$ by the following interpolation properties.
With the same notation as in Theorem \ref{thm:92}, for a grossencharacter $\varepsilon$ which factors $\GG = \Gal(L_{\infty}/K)$, we have
\begin{align}\label{eq:97}
&\Omega_p^{j-k} \cdot \varepsilon(\LL_{\Sigma, S}(L_{\infty}/K))\\
& \qquad = \Omega^{j-k} \cdot \left(\frac{\sqrt{d_K}}{2 \pi} \right)^j 
\cdot  G(\varepsilon) \cdot  \left(1-\frac{\varepsilon(\pe)}{p} \right)
\widehat{L}_{\Sigma \setminus \{\pe\}} (\varepsilon^{-1}, 0)
\times \left(\prod_{v \in S \setminus \{\pe\}} \frac{1 - \varepsilon(v) N(v)^{-1}}{1 - \varepsilon(v)} \right), 
\end{align}
where in the last product $v$ runs over places in $S \setminus \{\pe\}$ which are prime to the conductor of $\varepsilon$.

(2)
We define an invertible ideal $\iL_{\Sigma, S}(L_{\infty}/K)$ of $\RR$ by requiring
\[
\iL_{\Sigma, S}(L_{\infty}/K) \RRur = \LL_{\Sigma, S}(L_{\infty}/K) \RRur
\]
as invertible ideals of $\RRur$ (see Remark \ref{rem:113}(2) below for the existence; note also that, assuming the existence, we have the uniqueness by Lemma \ref{lem:90}).
We put $\iL_{\Sigma_0}(L_{\infty}/K) = \iL_{\Sigma, \Sigma_0}(L_{\infty}/K)$.
\end{defn}

\begin{rem}\label{rem:113}
We give remarks on the existences of $\LL$ and $\iL$ in Definition \ref{defn:96}.

(1) The existence of $\LL_{\Sigma_{\ff}, \{\pe\}}(L_{\ff, \infty}/K)$ follows immediately from Theorem \ref{thm:92}, since it is equal to $\mu_{\ff}$.
By taking the image under the natural map, for general $L/K$, the existence of $\LL_{\Sigma_{\ff}, \{\pe\}}(L_{\infty}/K)$ follows, where $\ff$ is the prime-to-$p$ component of the conductor of $L/K$.
Then the existence of $\LL_{\Sigma, S}(L_{\infty}/K)$ for general $\Sigma, S$ follows from the computation in Lemma \ref{lem:47} below (see \eqref{eq:100} and \eqref{eq:101}).

(2) The existence of $\iL_{\Sigma, S}(L_{\infty}/K)$ is not obvious at all.
We will prove the existence by the relation with the elliptic units, rather than by directly investigating the interpolation properties.
More concretely, the existence of $\iL_{\Sigma_{\ff}, \{\pe\}}(L_{\ff, \infty}/K)$ follows from Corollary \ref{cor:49}, and the general case follows by using the formulas corresponding to \eqref{eq:100} and \eqref{eq:101}.
\end{rem}

Finally we give the definition of one-variable ideals of $p$-adic $L$-functions.

\begin{defn}\label{defn:95}
Let $\ol{L_{\infty}}/L$ be any $\Z_p$-extension which is contained in the $\Z_p^2$-extension $L_{\infty}/L$.
Put $\ol{\RR} = \Z_p[[\Gal(\ol{L_{\infty}}/K)]]$.
We define the invertible ideal $\iL_{\Sigma, S}(\ol{L_{\infty}}/K)$ of $\ol{\RR}$ as the natural image of $\iL_{\Sigma, S}(L_{\infty}/K)$.
\end{defn}

For the well-definedness, see Corollary \ref{cor:73}.
By replacing $L$ by $L(E[\pe])$ in Definition \ref{defn:95}, the invertible ideal $\iL_{\Sigma, S}(L^{\pe}/K)$ of $\RR^{\pe}$, which appear in Theorem \ref{thm:75}, is defined.

\section{Determinant modules and Fitting ideals}\label{sec:02}

In this section, we give algebraic preliminaries required for the proof of Theorem \ref{thm:34}.

Let $R = \Z_p[[G]]$ be as in Subsection \ref{subsec:91}.
Let $\Inv_{R}$ denote the commutative group of invertible ideals of $R$.
Our main purpose in this section is to establish the following (the notation will be explained later).

\begin{thm}\label{thm:104}
We have a diagram
\begin{equation}\label{eq:06}
\xymatrix{
K_0(\PP_{R}) \ar[r] \ar[rd]_{\Fitt_{R}} &
K_0(\QQ_{R}) \ar[r]^-{\varphi} &
K_0(\DePTor(R)) \ar[ld]^{\Det_{R}} \\
& \Inv_{R} &\\
}
\end{equation}
which is anti-commutative, meaning that 
\[
\Det_{R} \circ \varphi = - \Fitt_{R}
\]
(in the additive notation).
Moreover, all maps in the diagram \eqref{eq:06} are isomorphic.
\end{thm}

This result might be more or less known to experts, but we give a proof for completeness.

\subsection{Fitting ideals $\Fitt_{R}$}

Following the notations in \cite{Kata_05}, we let $\PP_{R}$ (resp. $\QQ_{R}$) be the exact category consisting of finitely generated torsion $R$-modules $P$ (resp. $Q$) such that $\pd_{R}(P) \leq 1$ (resp. $\pd_{R}(Q) < \infty$).

For an (essentially small) exact category $\CC$ like $\PP_{R}$ and $\QQ_{R}$, we denote by $K_0(\CC)$ its Grothendieck group.
By definition, $K_0(\CC)$ has the following presentation by generators and relations:
the generators are $[X]$ for objects $X \in \CC$ and the relations are $[X] = [X'] + [X'']$ for exact sequences $0 \to X' \to X \to X'' \to 0$ in $\CC$.

\begin{thm}[{Resolution theorem}]\label{thm:102}
The natural group homomorphism $K_0(\PP_{R}) \to K_0(\QQ_{R})$ induced by the inclusion functor from $\PP_{R}$ to $\QQ_{R}$ is an isomorphism.
\end{thm}

\begin{proof}
See \cite[Theorem 3.1.13]{Ros94}, for example.
\end{proof}


\begin{prop}\label{prop:04}
$\Fitt_{R}$ induces a group isomorphism $\Fitt_{R}: K_0(\PP_{R}) \to \Inv_{R}$.
\end{prop}

\begin{proof}
See \cite[Proposition 2.7]{Kata_05} for the well-definedness and \cite[Remark 2.8]{Kata_05} for the injectivity.
To show the surjectivity, take any $I \in \Inv_{R}$.
We can take a non-zero-divisor $f \in R$ such that $fI \subset R$.
Since any invertible ideals are projective, we have $R/fI \in \PP_{R}$.
Then $\Fitt_{R}$ sends the element $[R/fI] - [R/fR] \in K_0(\PP_{R})$ to $I$.
\end{proof}

\hidden{
Let $R$ be a noetherian commutative ring.
Let $I \subset \Frac(R)$ be an invertible ideal, meaning that there is an ideal $J$ such that $IJ = R$.
Then $I$ is projective as an $R$-module.

To prove this, take $x_1, \dots, x_n \in I$ and $y_1, \dots, y_n \in J$ such that $\sum_i x_iy_i = 1$.
We define $I \to R^n$ by $x \mapsto (xy_1, \dots, xy_n)$ and $R^n \to I$ by $(a_1, \dots, a_n) \mapsto \sum_i a_ix_i$.
Then the composition is the identity on $I$, showing that $I$ is a direct summand of $R^n$.
}

\subsection{Determinant $\Det_{R}$}



Recall that a triangulated category $\CC$ is an additive category equipped with a translation functor $\CC \to \CC$, denoted by $X \mapsto X[1]$, and a notion of (distinguished) triangles
\[
X' \to X \to X'' \to X'[1] \qquad \text{(or, $X' \to X \to X'' \to$ for short)}
\]
satisfying a couple of axioms.

\hidden{
Here we review the octahedral axiom:
Let $X' \to X \to X'' \to X'[1]$ be a triangle and $Y \to X'$ a morphism.
Then an axiom claims that there are objects $Z', Z$ with triangles in the vertical lines of the commutative diagram
\[
\xymatrix{
Y \ar@{=}[r] \ar[d] &
Y \ar[d] & & \\
X' \ar[r] \ar[d] &
X \ar[r] \ar[d] &
X'' \ar[r] \ar@{=}[d] &
X'[1] \ar[d]\\
Z' \ar@{.>}[r] \ar[d] &
Z \ar@{.>}[r] \ar[d] &
X'' \ar[r] &
Z'[1]\\
Y[1] \ar@{=}[r] &
Y[1] &&&
}
\]
Now the octahedral axiom shows that there are dotted arrows which make the lower horizontal line a triangle.
}

For an (essentially small) triangulated category $\CC${\if0 like $\DePTor(R)$\fi}, we denote by $K_0(\CC)$ its Grothendieck group.
By definition, $K_0(\CC)$ has the following presentation by generators and relations:
the generators are $[X]$ for objects $X \in \CC$ and the relations are $[X] = [X'] + [X'']$ for triangles $X' \to X \to X'' \to $ in $\CC$.

We briefly introduce several categories arising from cochain complexes.

\begin{defn}
Let $\ChB(R)$ be the abelian category of bounded cochain complexes of $R$-modules.
A complex $F^{\bullet} \in \ChB(R)$ is said to be perfect if each $F^i$ is finitely generated and projective.
Let $\ChP(R)$ be the category of perfect complexes of $R$-modules.
Moreover, let $\ChPTor(R)$ be its subcategory consisting of those with torsion cohomology groups $H^i(F^{\bullet})$.

These categories $\ChB(R), \ChP(R)$, and $\ChPTor(R)$ are equipped with natural translation functors.
Let $\DeB(R)$, $\DeP(R)$ and $\DePTor(R)$ be the derived categories of $\ChB(R)$, $\ChP(R)$ and $\ChPTor(R)$, respectively.
Then these derived categories are triangulated categories.
\end{defn}

Now we introduce determinant modules.
We refer to \cite{KM76} for the detailed construction.
One can also refer to \cite[Subsection 3.1]{GKK_09}.

For a finitely generated projective $R$-module $F$, letting $\rank(F)$ denote the (locally constant) rank of $F$, we define the determinant of $F$ by
\[
\Det_{R}(F) = \bigwedge_{R}^{\rank(F)} F.
\]
Moreover, we denote its inverse by
\[
\Det_{R}^{-1}(F) = \Hom_{R}(\Det_{R}(F), R).
\]
More precisely, we should introduce the grade, but we omit it for simplicity.

For each complex $F^{\bullet} \in \ChP(R)$, we define its determinant by
\[
\Det_{R}(F^{\bullet}) 
= \bigotimes_{i \in \Z} \Det_{R}^{(-1)^i}(F^i).
\]
Similarly, we denote by $\Det_{R}^{-1}(F^{\bullet})$ its inverse.

Suppose $F^{\bullet} \in \ChPTor(R)$.
Since $\Frac(R) \otimes_{R} F^{\bullet}$ is acyclic, we have a natural isomorphism $\Det_{\Frac(R)}(\Frac(R) \otimes_{R} F^{\bullet}) \simeq \Frac(R)$.
Therefore, we have a natural injective map
\[
\Det_{R}(F^{\bullet}) \hookrightarrow \Det_{\Frac(R)} (\Frac(R) \otimes_{R} F^{\bullet}) \simeq \Frac(R).
\]
From now on, we identify $\Det_{R}(F^{\bullet})$ with its image in $\Frac(R)$.
Therefore, $\Det_{R}(F^{\bullet})$ is an invertible ideal of $R$.

\begin{lem}\label{lem:05}
$\Det_{R}$ induces a group homomorphism $\Det_{R}: K_0(\DePTor(R)) \to \Inv_{R}$.
\end{lem}

\begin{proof}
This is a standard fact.
\end{proof}

\subsection{Homomorphism $\varphi$}

We construct the homomorphism $\varphi$ in the diagram \eqref{eq:06}.

\begin{prop}\label{prop:03}
For each $Q \in \QQ_{R}$, take a projective resolution of $Q$, that is, a perfect complex $F_{Q}^{\bullet} \in \ChPTor(R)$ with $F_{Q}^i = 0$ for $i \geq 1$ and a quasi-isomorphism $F_{Q}^{\bullet} \to Q[0]$.
Then $[Q] \mapsto [F_{Q}^{\bullet}]$ gives a well-defined surjective homomorphism $\varphi: K_0(\QQ_{R}) \to K_0(\DePTor(R))$.
\end{prop}

\begin{proof}
It is a basic fact in homological algebra that two projective resolutions of the same module are homotopic to each other.
Moreover, for an exact sequence $0 \to Q' \to Q \to Q'' \to 0$ in $\QQ_{R}$, by the horseshoe lemma, we can take projective resolutions so that an exact sequence $0 \to F_{Q'}^{\bullet} \to F_{Q}^{\bullet} \to F_{Q''}^{\bullet} \to 0$ exists.
Therefore, we have a well-defined homomorphism $\varphi: K_0(\QQ_{R}) \to K_0(\DePTor(R))$.

We prove the surjectivity.
For $F^{\bullet} \in \DePTor(R)$, define $\length(F^{\bullet})$ by $\length(F^{\bullet}) = -\infty$ if $F^{\bullet}$ is acyclic, and
\[
\length(F^{\bullet}) = \max \{i \in \Z \mid H^i(F^{\bullet}) \neq 0\} - \min \{i \in \Z \mid H^i(F^{\bullet}) \neq 0\} \geq 0
\]
otherwise.
We shall show that $[F^{\bullet}] \in \Image(\varphi)$ by induction on $\length(F^{\bullet})$.
If $F^{\bullet}$ is acyclic, then $[F^{\bullet}] = 0$ and we have nothing to do.
Therefore, we may assume that $F^{\bullet}$ is not acyclic.
Since $[F^{\bullet}[n]] = (-1)^n[F^{\bullet}]$ for $n \in \Z$,
we may assume that $H^0(F^{\bullet}) \neq 0$ and $H^i(F^{\bullet}) = 0$ for $i \geq 1$.
Moreover, by truncation, we may assume that $F^i = 0$ for $i \geq 1$.

Suppose $\length(F^{\bullet}) = 0$.
Then the complexes $F^{\bullet}$ and $H^0(F^{\bullet})[0]$ are quasi-isomorphic.
By the definition of $\varphi$, this shows that $[F^{\bullet}] = \varphi([H^0(F^{\bullet})]) \in \Image(\varphi)$.

Suppose $\length(F^{\bullet}) \geq 1$.
Since the cohomology $H^0(F^{\bullet})$ is torsion, we can take a non-zero-divisor $f \in R$ which annihilates it.
By the projectivity of $F^0$, there exists a dotted arrow below which makes a commutative diagram
\[
\xymatrix{
F^{\bullet}:& \cdots \ar[r] & F^{-2} \ar[r] & F^{-1} \ar[r] & F^0 \ar[r] & 0\\
F_1^{\bullet}:& \cdots \ar[r]& 0 \ar[r] & fF^0 \ar@{.>}[u] \ar@{^{(}->}[r] & F^0 \ar[r] \ar@{=}[u] & 0\\
}
\]
This morphism $F_1^{\bullet} \to F^{\bullet}$ in $\DePTor(R)$ gives an object $F_2^{\bullet}$ of $\DePTor(R)$ which fits in a triangle $F_1^{\bullet} \to F^{\bullet} \to F_2^{\bullet} \to $.
We have $\length(F_1^{\bullet}) \leq 0$ and $H^0(F_1^{\bullet}) \to H^0(F^{\bullet})$ is surjective by construction.
Therefore, by the induced long exact sequence, we can see that $\length(F_2^{\bullet}) < \length(F^{\bullet})$.
By the induction hypothesis, we obtain
\[
[F^{\bullet}] = [F_1^{\bullet}] + [F_2^{\bullet}] \in \Image(\varphi).
\]
This completes the proof of the surjectivity of $\varphi$.
\end{proof}

\begin{prop}\label{prop:30}
The diagram \eqref{eq:06} is anti-commutative.
\end{prop}

\begin{proof}
For any $P \in \PP_{R}$, we can take a free resolution $0 \to R^a \overset{h}{\to} R^a \to P \to 0$ of $P$.
Then, by definition, we have 
\[
\varphi([P]) = [\cdots \to 0 \to R^a \overset{h}{\to} R^a \to 0 \to \cdots]
\]
which concentrates at the degrees $-1, 0$.
Now the assertion $\Fitt_R(P)^{-1} = \Det_{R}(\varphi([P]))$ follows from a standard computation as in \cite[Lemma 3.8]{GKK_09}, for example.
\end{proof}

Now Theorem \ref{thm:104} follows from Theorem \ref{thm:102}, Proposition \ref{prop:04}, Lemma \ref{lem:05}, Proposition \ref{prop:03}, and Proposition \ref{prop:30}.

\begin{rem}\label{rem:112}
By Theorem \ref{thm:82}, we can complement Theorem \ref{thm:104} as

\begin{equation}\label{eq:29}
\xymatrix{
K_0(\PP_{R}) \ar[r] \ar[rd]_{\Fitt_{R}} &
K_0(\QQ_{R}) \ar[r]^-{\varphi} \ar[d]^{\SF{0}{R}} &
K_0(\DePTor(R)) \ar[ld]^{\Det_{R}} \\
& \Inv_{R} &\\
}
\end{equation}
where all maps are isomorphic, the left triangle is commutative, and the right triangle is anti-commutative.
\end{rem}

\section{Arithmetic complexes}\label{sec:105}

In this section, we review facts on Galois cohomology complexes (this section does not have novelty).
We follow the notations in \cite{GKK_09}, and refer to Nekov\'{a}\v{r} \cite{Nek06} for detail.

Keep the notation of Section \ref{sec:01}.
We denote by $K_{\Sigma}$ the maximal algebraic extension of $K$ unramified outside $\Sigma$.
%
Let $\Z_p(1) = \varprojlim_{m} \mu_{p^m}$ be the Tate module.
We set
\[
\bT = \Z_p(1) \otimes_{\Z_p} \RR,
\]
which is a free $\RR$-module of rank $1$.
We equip $\bT$ with a $\Gal(\ol{K}/K)$-action, where the action on the second component $\RR$ is the inverse of the natural group homomorphism $\kappa: \Gal(\overline{K}/K) \twoheadrightarrow \Gal(L_{\infty}/K) \hookrightarrow \RR^{\times}$.
By the Shapiro lemma, we have isomorphisms such as
\begin{align}\label{eq:121}
H^i(K_{\Sigma}/K, \bT) \simeq \varprojlim_{K'} H^i(K_{\Sigma}/K', \Z_p(1)),\\
H^i(K_{\Sigma}/K, \bT^{\dual}(1)) \simeq H^i(K_{\Sigma}/L_{\infty}, \Q_p/\Z_p),
\end{align}
where $K'$ runs over finite extensions of $K$ in $L_{\infty}$ and the projective limit is taken with respect to the corestriction maps.

By the cochain complex construction (see \cite[(3.4.1)]{Nek06}), we obtain complexes like
\[
\derR\Gamma(K_{\Sigma}/K, \bT), 
\qquad \derR\Gamma(K_{\Sigma}/K, \bT^{\dual}(1))^{\dual},
\]
where $(-)^{\dual}$ denotes the Pontryagin dual (also for complexes), and the local counterparts of them.
We summarize properties of these complexes.

\begin{prop}\label{prop:13}
(1)
The complexes
\[
\derR\Gamma(K_{\Sigma}/K, \bT),
\qquad \derR\Gamma(K_{\Sigma}/K, \bT^{\dual}(1))^{\dual},
\]
and
\[
\derR\Gamma(K_v, \bT),
\qquad \derR\Gamma(K_v, \bT^{\dual}(1))^{\dual}
\]
for any finite place $v$ of $K$, are perfect.
That is, these complexes are contained in $\DeP(\RR)$.

(2)
We have a natural isomorphism
\[
\derR\Gamma(K_v, \bT) \simeq \derR\Gamma(K_v, \bT^{\dual}(1))^{\dual}[-2]
\]
for any finite place $v$ of $K$.

(3)
We have a triangle 
\begin{equation}\label{eq:07}
\derR\Gamma(K_{\Sigma}/K, \bT) \to \bigoplus_{v \in \Sigma} \derR\Gamma(K_v, \bT) \to \derR\Gamma(K_{\Sigma}/K, \bT^{\dual}(1))^{\dual}[-2] \to 
\end{equation}
in $\DeP(\RR)$.
\end{prop}

\begin{proof}
(1) For $\derR\Gamma(K_{\Sigma}/K, \bT)$ and $\derR\Gamma(K_v, \bT)$, see \cite[Proposition (4.2.9)]{Nek06}.
The others follow from (2) and (3) below.

(2) This is the local Tate duality \cite[Proposition (5.2.4)]{Nek06}.

(3) This is the Poitou-Tate duality \cite[Proposition (5.4.3)]{Nek06}.
\end{proof}

Recall that, in our case, the weak Leopoldt conjecture $H^2(K_{\Sigma}/K, \bT^{\dual}(1)) = 0$ has been proved.
Then the long exact sequence associated to \eqref{eq:07} is
\begin{align}\label{eq:10}
0 & \to
H^1(K_{\Sigma}/K, \bT) \to
\bigoplus_{v \in \Sigma} H^1(K_v, \bT) \to
H^1(K_{\Sigma}/K, \bT^{\dual}(1))^{\dual}\\
& \to 
H^2(K_{\Sigma}/K, \bT) \to
\bigoplus_{v \in \Sigma} H^2(K_v, \bT) \to
H^0(K_{\Sigma}/K, \bT^{\dual}(1))^{\dual} \to
0.
\end{align}
By using \eqref{eq:121}, this sequence \eqref{eq:10} is regarded as the projective limit of the usual Poitou-Tate long exact sequences.



For a finite place $v$ of $K$, let $\GG_v \subset \GG$ be the decomposition group of $v$ in $L_{\infty}/K$.
In the notation in Definition \ref{defn:52}, we have $\GG_v = D_v(L_{\infty}/K)$.
Put 
\[
Z_v(L_{\infty}/K) = \Z_p[[\GG/\GG_v]]
\]
as in Definition \ref{defn:52}.

We summarize well-known descriptions of the cohomology groups in \eqref{eq:10}, using \eqref{eq:121}.

\begin{lem}\label{lem:12}
(1)
We have
\[
H^i(K_{\Sigma}/K, \bT^{\dual}(1))^{\dual} 
\simeq \begin{cases}
	\Z_p & (i = 0)\\
	X_{\Sigma}(L_{\infty}) & (i=1)\\
	0 & (i \neq 0, 1).
\end{cases}
\]

(2)
For any finite place $v$ of $K$, we have a natural isomorphism
\[
H^2(K_v, \bT) \simeq Z_v(L_{\infty}/K).
\]

(3)
We have natural isomorphisms
\[
H^1(K_{\Sigma}/K, \bT) \simeq \varprojlim_{K'} ((\OO_{K', \Sigma}^{\times})^{\wedge})
\]
and
\[
H^1(K_v, \bT) \simeq \varprojlim_{K'} (((K' \otimes_K K_v)^{\times})^{\wedge})
\]
for any finite place $v$, where $(-)^{\wedge}$ denotes the $p$-adic completion, $K'$ runs over finite extensions of $K$ in $L_{\infty}$, $\OO_{K', \Sigma}$ is the ring of $\Sigma$-integers of $K'$, and the inverse limit is taken with respect to the norm maps.
\end{lem}

%
%


The following is a special phenomenon in our setting.
It is essentially \cite[Lemma 5.14]{Kata_05}.

\begin{lem}\label{lem:11}
If $v \mid p$, then $\pd_{\RR}(H^1(K_v, \bT)) < \infty$ and $\pd_{\RR}(H^2(K_v, \bT)) < \infty$.
\end{lem}

\begin{proof}
The assertion for $H^2$ follows from Remark \ref{rem:85} and Lemma \ref{lem:12}(2).
Then the assertion for $H^1$ follows since the complex is perfect.
\end{proof}

\section{Fitting ideals of Iwasawa modules}\label{sec:107}

In this section, we prove Theorem \ref{thm:34}.
Let $S$ and $\Sigma$ be as in Theorem \ref{thm:34}.
For readability, we omit $L_{\infty}/K$ from the notation when no confusion can occur; for example, $X_S = X_S(L_{\infty})$, $Z_{\Sigma \setminus S}^0 = Z_{\Sigma \setminus S}^0(L_{\infty}/K)$, and $\iL_{\Sigma, S} = \iL_{\Sigma, S}(L_{\infty}/K)$.

\subsection{Complex $C_{\Sigma, S}$}
In this subsection, we define and study a complex $C_{\Sigma, S}$ which will play an important role in the proof of Theorem \ref{thm:34}.
The idea behind the definition is the same as \cite{GKK_09} in the one-variable case.

\begin{defn}\label{defn:36}
Using the second morphism in \eqref{eq:07}, we construct $C_{\Sigma, S} = C_{\Sigma, S}(L_{\infty}/K)$ which fits in a triangle
\begin{equation}\label{eq:40}
C_{\Sigma, S} 
\to \bigoplus_{v \in \Sigma \setminus S} \derR\Gamma(K_v, \bT) 
\to \derR\Gamma(K_{\Sigma}/K, \bT^{\dual}(1))^{\dual}[-2] \to.
\end{equation}
Note that, in the direct sum, $v$ also takes the value $\ol{\pe}$.
\end{defn}

By the triangles \eqref{eq:07} and \eqref{eq:40}, the complex $C_{\Sigma, S}$ also fits in the following triangle
\begin{equation}\label{eq:14}
C_{\Sigma, S} \to \derR\Gamma(K_{\Sigma}/K, \bT) \to \bigoplus_{v \in S} \derR\Gamma(K_v, \bT) \to.
\end{equation}

\begin{prop}\label{prop:37}
The complex $C_{\Sigma, S}$ is in $\DePTor(\RR)$.
We have $H^i(C_{\Sigma, S}) = 0$ unless $i = 2$, and we have an exact sequence
\begin{equation}\label{eq:39}
0 \to X_{S} \to H^2(C_{\Sigma, S}) \to Z_{\Sigma \setminus S}^0 \to 0.
\end{equation}
\end{prop}

\begin{proof}
By Proposition \ref{prop:13}(1), the complex $C_{\Sigma, S}$ is in $\DeP(\RR)$.
The triangle \eqref{eq:40} induces a long exact sequence
\begin{align}
0 & \to H^1(C_{\Sigma, S}) 
 \to \bigoplus_{v \in \Sigma \setminus S} H^1(K_{v}, \bT) 
 \to H^1(K_{\Sigma}/K, \bT^{\dual}(1))^{\dual} \\
& \to H^2(C_{\Sigma, S})
 \to \bigoplus_{v \in \Sigma \setminus S} H^2(K_{v}, \bT) 
\to H^0(K_{\Sigma}/K, \bT^{\dual}(1))^{\dual} 
\to H^3(C_{\Sigma, S}) 
\to 0.
\end{align}
Using the validity of the $\overline{\pe}$-adic Leopoldt conjecture, we have an exact sequence
\[
0 \to \bigoplus_{v \in \Sigma \setminus S} H^1(K_{v}, \bT) \to X_{\Sigma} \to X_{S} \to 0.
\]
Then by the descriptions in Lemma \ref{lem:12}, we obtain $H^i(C_{\Sigma, S}) = 0$ for $i \neq 2$ and the exact sequence \eqref{eq:39}.
Finally, \eqref{eq:39} shows that $H^2(C_{\Sigma, S})$ is torsion, namely $C_{\Sigma, S}$ is in $\DePTor(\RR)$.
\end{proof}

\begin{cor}\label{cor:41}
We have
\[
\Fitt_{\RR}(X_{S}) = \Det_{\RR}^{-1}(C_{\Sigma, S}) \Fitt_{\RR}^{[1]}(Z_{\Sigma \setminus S}^0).
\]
\end{cor}

\begin{proof}
By the properties of $\SF{n}{\RR}$, the exact sequence \eqref{eq:39} shows
\[
\SF{0}{\RR}(X_{S}) = \SF{0}{\RR}(H^2(C_{\Sigma, S})) \SF{1}{\RR}(Z_{\Sigma \setminus S}^0).
\]
We observe the following.
\begin{itemize}
\item $\SF{0}{\RR}(X_{S}) = \Fitt_{\RR}(X_S)$ by $\pd_{\Lambda}(X_{S}) \leq 1$ (cf. \cite[Proposition 5.15]{Kata_05}), where $\Lambda = \Z_p[[\Gal(L_{\infty}/L)]]$ plays the role of $R'$ in Subsection \ref{subsec:91}.
\item $\SF{0}{\RR}(H^2(C_{\Sigma, S})) = \Det_{\RR}^{-1}(C_{\Sigma, S})$ by Theorem \ref{thm:104} (actually by Remark \ref{rem:112}).
\item $\SF{1}{\RR}(Z_{\Sigma \setminus S}^0) = \vSF{1}{\RR}(Z_{\Sigma \setminus S}^0)$ by Lemma \ref{lem:83}.
\end{itemize}
Therefore, the corollary follows.
\end{proof}

By Corollary \ref{cor:41}, the proof of Theorem \ref{thm:34} reduces to showing the following.

\begin{thm}\label{thm:42}
We have
\[
\Det_{\RR}^{-1}(C_{\Sigma, S}) = \iL_{\Sigma, S}.
\]
\end{thm}

In the rest of this section, we prove Theorem \ref{thm:42}.


\subsection{Reduction to special cases}\label{subsec:45}

In this subsection, by computing local factors, we reduce the proof of Theorem \ref{thm:42} to special cases.

We recall the description of the local factor $\Det_{\RR}(\derR\Gamma(K_v, \bT))$.
For each finite place $v$ of $K$ outside $p$, let $\TT_v \subset \GG$ be the inertia subgroup and $\sigma_v \in \GG/\TT_v$ the Frobenius automorphism.

\begin{prop}[{\cite[Proposition 3.13]{GKK_09}}]\label{prop:43}
For every finite place $v$ of $K$ outside $p$, there exists 
a unique element $f_v \in \Frac(\RR)^{\times}$ satisfying the following.

(i) We have
\[
\Det_{\RR}(\derR\Gamma(K_v, \bT)) = (f_v).
\]

(ii) For any continuous character $\varepsilon: \GG \to \overline{\Q_p}^{\times}$ 
which is nontrivial on $\GG_v$, we have
\[
\varepsilon(f_v) =
\begin{cases}
	\frac{1 - \varepsilon(\sigma_v)N(v)^{-1}}{1 - \varepsilon(\sigma_v)} & 
	      \text{if $\varepsilon$ is unramified at $v$;}\\
	    1 & 
			  \text{if $\varepsilon$ is ramified at $v$.}
\end{cases}
\]
\end{prop}

\begin{proof}
This is proved in \cite{GKK_09}.
Though \cite{GKK_09} treats the one-variable case, this proposition is essentially a local statement and indeed we find $f_v$ in $\Frac(\Z_p[[\GG_v]])$.
\end{proof}

Let $v$ be a finite place of $K$ outside $p$ which is unramified in $L/K$.
We denote by $K_v^{\ur}$ the maximal unramified extension of $K_v$.
Then the absolute Galois group of $K_v^{\ur}$ acts on $\bT$ trivially.
We define a complex $\derR\Gamma_{/f}(K_v, \bT)$ by postulating a triangle (cf. \cite[(7.1.2)]{Nek06})
\[
\derR\Gamma(K_v^{\ur}/K_v, \bT) \overset{\Inf}{\to} \derR\Gamma(K_v, \bT) \to \derR\Gamma_{/f}(K_v, \bT) \to.
\]

\begin{prop}\label{prop:88}
For every finite place $v$ of $K$ outside $p$ which is unramified in $L/K$, 
we have
\[
\Det_{\RR}(\derR\Gamma_{/f}(K_v, \bT)) = (1 - \sigma_v^{-1})^{-1}.
\]
\end{prop}

\begin{proof}
The cohomology groups of $\derR\Gamma_{/f}(K_v, \bT)$ vanish, except for 
\[
H^2(\derR\Gamma_{/f}(K_v, \bT)) \simeq Z_v \simeq \RR / (1 - \sigma_v)
\]
as in Lemma \ref{lem:12}(2).
Thus the assertion follows from Theorem \ref{thm:104}.
\end{proof}

By applying the local computations in Propositions \ref{prop:43} and \ref{prop:88}, we show the following.

\begin{lem}\label{lem:47}
The assertion of Theorem \ref{thm:42} does not depend on the choice of $\Sigma$ and $S$.
\end{lem}

\begin{proof}
First we show the independency from $S$.
On the algebraic side, by comparing the definitions in Definition \ref{defn:36}, we obtain
\[
[C_{\Sigma, S}] = - \sum_{v \in S \setminus \{\pe\}} [\derR\Gamma(K_v, \bT)] + [C_{\Sigma, \{\pe\}}]
\] 
in $K_0(\DePTor(\RR))$.
On the analytic side, by Definition \ref{defn:96}, we have
\begin{equation}\label{eq:100}
\LL_{\Sigma, S} = \left( \prod_{v \in S \setminus \{\pe\}} f_v \right) \LL_{\Sigma, \{\pe\}}
\end{equation}
as invertible ideals of $\RR$, where $f_v$ is as in Proposition \ref{prop:43}.
In particular, the similar relation holds for the ideals $\iL_{\Sigma, S}$ and $\iL_{\Sigma, \{\pe\}}$.
Therefore, by Proposition \ref{prop:43}, the assertion for $S$ is equivalent to that for $\{\pe\}$.

Next we show the independency from $\Sigma$.
Let $\Sigma'$ be a finite set of finite places of $K$ which contains $\Sigma$.
Then we have a triangle
\[
\derR\Gamma(K_{\Sigma}/K, \bT) \to \derR\Gamma(K_{\Sigma'}/K, \bT) 
\to \bigoplus_{v \in \Sigma' \setminus \Sigma} \derR\Gamma_{/f}(K_v, \bT) \to
\]
by \cite[Proposition (7.8.8)]{Nek06}.
Combining with the triangles \eqref{eq:14} for $\Sigma$ and $\Sigma'$ with $S = \{\pe\}$, we have
\[
[C_{\Sigma', \{\pe\}}]
=  \sum_{v \in \Sigma' \setminus \Sigma} [\derR\Gamma_{/f}(K_v, \bT)]
+ [C_{\Sigma, \{\pe\}}].
\]
On the analytic side, by Definition \ref{defn:96}, we have
\begin{equation}\label{eq:101}
\LL_{\Sigma', \{\pe\}} =  
\left( \prod_{v \in \Sigma' \setminus \Sigma} (1 - \sigma_v^{-1}) \right)
\LL_{\Sigma, \{\pe\}},
\end{equation}
and the similar formula for $\iL_{\Sigma', \{\pe\}}$ and $\iL_{\Sigma, \{\pe\}}$.
Therefore, by Proposition \ref{prop:88}, the independency follows.
\end{proof}

\begin{lem}\label{lem:46}
Let $L'_{\infty}/K$ be an abelian extension which is a finite extension of  $L_{\infty}$.
Then Theorem \ref{thm:42} for $L'_{\infty}/K$ implies Theorem \ref{thm:42} for $L_{\infty}/K$.
\end{lem}

\begin{proof}
By Lemma \ref{lem:47}, we may assume that the same $\Sigma$ and $S$ are chosen for both $L_{\infty}'/K$ and $L_{\infty}/K$.
We put $\RR' = \Z_p[[\Gal(L'_{\infty}/K)]]$.
The canonical map $\RR' \to \RR$ induces a map $\Frac(\RR') \to \Frac(\RR)$, which we denote by $\pi_{L'_{\infty}/L_{\infty}}$.
By \cite[Proposition 1.6.5(3)]{FK06}, we see that $C_{\Sigma, S}(L'_{\infty}/K) \otimes_{\RR'} \RR \simeq C_{\Sigma, S}(L_{\infty}/K)$.
This implies that 
\[
\pi_{L'_{\infty}/L_{\infty}}(\Det_{\RR'}^{-1}(C_{\Sigma, S}(L'_{\infty}/K))) 
= \Det_{\RR}^{-1}(C_{\Sigma, S}(L_{\infty}/K))
\]
as invertible ideals of $\RR$.
On the other hand, it is directly shown by the interpolation property that 
\[
\pi_{L'_{\infty}/L_{\infty}}(\iL_{\Sigma, S}(L'_{\infty}/K)) = \iL_{\Sigma, S}(L_{\infty}/K).
\]
Thus the lemma follows.
\end{proof}

\subsection{Elliptic units, equivariant main conjecture, and $p$-adic $L$-function}

By Lemmas \ref{lem:47} and \ref{lem:46}, in order to prove Theorem \ref{thm:42}, 
we may focus on the following situation:
\[
L = L_{\ff} = K(\ff p), \qquad \Sigma = \Sigma_{\ff},\qquad S = \{\pe\}
\]
for a fixed nonzero integral ideal $\ff$ of $K$ which is prime to $p$, where the notation is introduced just before Theorem \ref{thm:92}.
We also recall that $\GG_{\ff} = \Gal(L_{\ff, \infty}/K)$ and $\RR_{\ff} = \Z_p[[\GG_{\ff}]]$.
Let $\bT_{\ff}$ be the associated Galois representation of $\Gal(\overline{K}/K)$ over $\RR_{\ff}$.

As in \cite[pages 100--101]{JK11}, let $\zeta(\ff) \in \Frac(\RR_{\ff}) \otimes_{\RR_{\ff}} H^1(K_{\Sigma_{\ff}}/K, \bT_{\ff})$ be the element constructed by the elliptic units, where we use the identification in Lemma \ref{lem:12}(3).
Let $\Delta_{\ff}$ be the torsion subgroup of $\GG_{\ff}$.
Fix a splitting of $\GG_{\ff} \twoheadrightarrow \GG_{\ff}/\Delta_{\ff} \simeq \Z_p^2$, and we put $\sI_{\ff} = \Ker(\RR_{\ff} \twoheadrightarrow \Z_p[\Delta_{\ff}])$.
The denominator of $\zeta({\ff})$ is quite small; in fact, we have
\[
\EE_{\ff} := \sI_{\ff} \zeta(\ff) \subset H^1(K_{\Sigma_{\ff}}/K, \bT_{\ff}).
\]
Moreover, we know that $H^1(K_{\Sigma_{\ff}}/K, \bT_{\ff})/\EE_{\ff}$ is torsion (\cite[Theorem 5.7]{JK11}) and $\pd_{\RR_{\ff}}(\EE_{\ff}) < \infty$.

We define complexes $D_{\Sigma_{\ff}, \EE_{\ff}}^{\glob}, D_{\pe, \EE_{\ff}}^{\local}$ in $\DeP(\RR_{\ff})$ which admit triangles
\[
\EE_{\ff}[-1] \to \derR\Gamma(K_{\Sigma_{\ff}}/K, \bT_{\ff}) \to D_{\Sigma_{\ff}, \EE_{\ff}}^{\glob} \to
\]
and
\[
\EE_{\ff}[-1] \to \derR\Gamma(K_{\pe}, \bT_{\ff}) \to D_{\pe, \EE_{\ff}}^{\local} \to.
\]
Then \eqref{eq:14} induces a triangle
\begin{equation}\label{eq:15}
C_{\Sigma_{\ff}, \{\pe\}} \to D_{\Sigma_{\ff}, \EE_{\ff}}^{\glob} \to D_{\pe, \EE_{\ff}}^{\local} \to.
\end{equation}
By construction, these complexes live in $\DePTor(\RR_{\ff})$.
Hence we obtain
\begin{equation}\label{eq:98}
- [C_{\Sigma_{\ff}, \{\pe\}}] = [D_{\pe, \EE_{\ff}}^{\local}] - [D_{\Sigma_{\ff}, \EE_{\ff}}^{\glob}]
\end{equation}
in $K_0(\DePTor(\RR_{\ff}))$.

The global contribution in \eqref{eq:98} is exactly what the equivariant main conjecture \cite{JK11} describes.

\begin{thm}[{\cite[Theorem 5.7 and Corollary 5.12]{JK11}}]\label{thm:48}
We have 
\[
\Det_{\RR_{\ff}}(D_{\Sigma_{\ff}, \EE_{\ff}}^{\glob}) = \RR_{\ff}.
\]
In other words (thanks to Theorem \ref{thm:104}), we have $[D_{\Sigma_{\ff}, \EE_{\ff}}^{\glob}] = 0$ in $K_0(\DePTor(\RR_{\ff}))$.
\end{thm}

For the local contribution in \eqref{eq:98}, we first observe the connection between elliptic units and the $p$-adic $L$-functions.

\begin{thm}\label{thm:94}
We have 
\[
\SF{0}{\RR_{\ff}}(H^1(K_{\pe}, \bT_{\ff}) / \EE_{\ff}) = \iL_{\Sigma_{\ff}, \{\pe\}}.
\]
\end{thm}

\begin{proof}
By \cite[Proposition III.1.3]{dS87}, we have an exact sequence of $\RRur_{\ff}$-modules
\[
0 \to \Zpur \widehat{\otimes}_{\Z_p} H^1(K_{\pe}, \bT_{\ff}) 
\overset{i}{\to} \RRur_{\ff} \to \Zpur \otimes_{\Z_p} Z_{\pe} (L_{\ff, \infty}/K)(1) \to 0,
\] 
where $i$ is constructed via Coleman power series.
Moreover, by (the proof of) \cite[Theorem II. 4.14]{dS87}, we have $i(\zeta(\ff)) = 12 \mu_{\ff}$.
Note that the coefficient $12$ comes from the final paragraph of the proof of \cite[Theorem II. 4,12]{dS87}, but it does not matter since we are assuming $p \geq 5$.

When we are concerned with $\SF{0}{\RR_{\ff}}$, by Proposition \ref{prop:22}, we can ignore the pseudo-null modules $Z_{\pe} (L_{\ff, \infty}/K)(1)$ and $\RR_{\ff}/\sI_{\ff} \simeq \Z_p[\Delta_{\ff}]$.
Hence we obtain
\[
\SF{0}{\RR_{\ff}}(H^1(K_{\pe}, \bT_{\ff}) / \EE_{\ff}) \RRur_{\ff}
= \Fitt_{\RRur_{\ff}} (\RRur_{\ff} / i(\zeta(\ff)))
 = \mu_{\ff} \RRur_{\ff}
 = \LL_{\Sigma_{\ff}, \{\pe\}} \RRur_{\ff},
\]
where the final equation is by definition (see Remark \ref{rem:113}(1)).
By Definition \ref{defn:96}, this completes the proof.
\end{proof}

\begin{cor}\label{cor:49}
We have
\[
\Det_{\RR_{\ff}}(D_{\pe, \EE_{\ff}}^{\local}) = \iL_{\Sigma_{\ff}, \{\pe\}}.
\]
\end{cor}

\begin{proof}
By the definition of $D_{\pe, \EE_{\ff}}^{\local}$, we have
\begin{align}
H^1(D_{\pe, \EE_{\ff}}^{\local}) \simeq H^1(K_{\pe}, \bT_{\ff}) / \EE_{\ff},\\
H^2(D_{\pe, \EE_{\ff}}^{\local}) \simeq H^2(K_{\pe}, \bT_{\ff}) \simeq Z_{\pe}(L_{\ff, \infty}/K),
\end{align}
and the other cohomology groups are zero.
By Remark \ref{rem:85}, we have $\pd_{\RR_{\ff}}(H^2(D_{\pe, \EE_{\ff}}^{\local})) < \infty$, so 
\[
\Det_{\RR_{\ff}}(D_{\pe, \EE_{\ff}}^{\local})
= \Det_{\RR_{\ff}} \left(H^1(D_{\pe, \EE_{\ff}}^{\local})[-1] \right) \Det_{\RR_{\ff}} \left(H^2(D_{\pe, \EE_{\ff}}^{\local})[-2] \right).
\]
Applying Theorem \ref{thm:104} and the equation \eqref{eq:108}, we obtain
\[
\Det_{\RR_{\ff}}(D_{\pe, \EE_{\ff}}^{\local})
= \SF{0}{\RR_{\ff}}(H^1(K_{\pe}, \bT_{\ff}) / \EE_{\ff}).
\]
Hence Theorem \ref{thm:94} implies the corollary.
\end{proof}

Now \eqref{eq:98}, Theorem \ref{thm:48}, and Corollary \ref{cor:49} complete the proof of Theorem \ref{thm:42}.
This also completes the proofs of Theorems \ref{thm:33} and \ref{thm:34}.

\section{Application to CM elliptic curves}\label{sec:106}

In this section, we prove Theorem \ref{thm:75} by using Theorem \ref{thm:42}.
We use the same notation as in Theorem \ref{thm:75}.

Here we outline the proof of Theorem \ref{thm:75}.
For simplicity, assume $L(E[\pe]) = L$ in this outline.
Then $L^{\pe} = L(E[\pe^{\infty}])$ is a $\Z_p$-extension of $L$ contained in $L_{\infty}$.

\begin{itemize}
\item[(i)] Subsection \ref{subsec:67} is a descent part from $X_{S}(L_{\infty})$ to $X_{S}(L^{\pe})$.
More precisely, from Theorem \ref{thm:42} on the complex $C_{\Sigma, S}(L_{\infty}/K)$, we determine the ideal
\[
\Fitt_{\RR^{\pe}}(X_{S}(L^{\pe})).
\]
\item[(ii)] In Subsection \ref{subsec:68}, we connect $X_{S}(L^{\pe})$ and $\Sel_{S}(E/L^{\pe})[\pe^{\infty}]$.
More precisely, we show
\[
\Sel_{S}(E/L^{\pe})[\pe^{\infty}] \simeq \Hom(X_{S}(L^{\pe}), E[\pe^{\infty}]).
\]
\item[(iii)] Subsection \ref{subsec:69} is a descent part from $\Sel_{S}(E/L^{\pe})[\pe^{\infty}]$ to $\Sel_{S}(E/L)[\pe^{\infty}]$.
Indeed, we observe an exact control theorem
\[
\Sel_{S}(E/L)[\pe^{\infty}] \overset{\sim}{\to} \Sel_{S}(E/L^{\pe})[\pe^{\infty}]^{\Gal(L^{\pe}/L)}.
\]
\end{itemize}

In Subsection \ref{subsec:74}, we deduce Theorem \ref{thm:75} from these results (i), (ii), and (iii).

\subsection{Specialization to one-variable}\label{subsec:67}

This subsection does not concern an elliptic curve and can deal with general $\Z_p$-extensions.

As in Definition \ref{defn:95}, let $\ol{L_{\infty}}/L$ be any $\Z_p$-extension contained in $L_{\infty}$.
Put $\ol{\RR} = \Z_p[[\Gal(\ol{L_{\infty}}/K)]]$.
Let $\ol{\bT} = \Z_p(1) \otimes_{\Z_p} \ol{\RR}$ be the associated Galois representation of $\Gal(\overline{K}/K)$ over $\ol{\RR}$.

Using the perfect complex $C_{\Sigma, S} = C_{\Sigma, S}(L_{\infty}/K)$ in Definition \ref{defn:36},
we define a perfect complex $\ol{C_{\Sigma, S}} = C_{\Sigma, S}(\ol{L_{\infty}}/K)$ by
\[
\ol{C_{\Sigma, S}} = C_{\Sigma, S} \otimesL_{\RR} \ol{\RR}.
\]
Then by \cite[Proposition 1.6.5(3)]{FK06}, the triangle \eqref{eq:40} yields a triangle
\[
\ol{C_{\Sigma, S}} 
\to \bigoplus_{v \in \Sigma \setminus S} \derR\Gamma(K_v, \ol{\bT}) 
\to \derR\Gamma(K_{\Sigma}/K, \ol{\bT}^{\dual}(1))^{\dual}[-2] \to.
\]

Exactly as in Proposition \ref{prop:37}, we obtain the following.

\begin{prop}\label{prop:62}
The complex $\ol{C_{\Sigma, S}}$ is in $\DePTor(\ol{\RR})$.
We have $H^i(\ol{C_{\Sigma, S}}) = 0$ unless $i = 2$, and we have an exact sequence
\begin{equation}\label{eq:65}
0 \to X_{S}(\ol{L_{\infty}}) \to H^2(\ol{C_{\Sigma, S}}) \to Z_{\Sigma \setminus S}^0(\ol{L_{\infty}}/K) \to 0.
\end{equation}
\end{prop}


By Theorem \ref{thm:42}, we deduce the following.

\begin{cor}\label{cor:73}
We have
\[
\Fitt_{\ol{\RR}}(X_{S}(\ol{L_{\infty}})) 
= \iL_{\Sigma, S}(\ol{L_{\infty}}/K) \Fitt_{\ol{\RR}}^{[1]}(Z_{\Sigma \setminus S}^0(\ol{L_{\infty}}/K)).
\]
\end{cor}

\begin{proof}
By Theorem \ref{thm:42} and the first assertion in Proposition \ref{prop:62}, the image $\iL_{\Sigma, S}(\ol{L_{\infty}}/K)$ of $\iL_{\Sigma, S}(L_{\infty}/K)$ in $\Frac(\ol{\RR})$ is indeed a well-defined invertible ideal, and we have 
\[
\Det_{\ol{\RR}}^{-1}(\ol{C_{\Sigma, S}}) = \iL_{\Sigma, S}(\ol{L_{\infty}}/K).
\]
The rest of the proof is exactly similar to Corollary \ref{cor:41}.
\end{proof}

\begin{rem}\label{rem:100}
While Corollary \ref{cor:73} is analogous to Corollary \ref{cor:41}, we do not have an analogue of Lemma \ref{lem:79} for $\ol{L_{\infty}}/K$.
Indeed, as in Remark \ref{rem:86}, the module $Z_{\pe}(\ol{L_{\infty}}/K)$ is not pseudo-null.
What is even worse is that we do not have $\pd_{\ol{\RR}}(Z_{\pe}(\ol{L_{\infty}}/K)) < \infty$ in general.
Thus the analogue of \eqref{eq:108} does not hold.
\end{rem}

\begin{rem}
It seems impossible to deduce Corollary \ref{cor:73} directly from Theorem \ref{thm:34} because we do not have an exact control theorem between $X_S(L_{\infty})$ and $X_S(\ol{L_{\infty}})$.
Remark \ref{rem:100} also implies that the descent is hard on the right hand sides.
\end{rem}

\subsection{Iwasawa modules and Selmer groups}\label{subsec:68}

First we define the Selmer groups.
Recall that $E$ has two commutative actions by $\OO_K$ and by the absolute Galois group of $K$.
In particular, we have a decomposition
\[
E[p^{\infty}] = E[\pe^{\infty}] \oplus E[\ol{\pe}^{\infty}]
\]
as a Galois representation.
We will be particularly interested in $E[\pe^{\infty}]$.

\begin{defn}\label{defn:66}
For any abelian extension $\FF$ of $K$, we define the $S$-imprimitive Selmer group $\Sel_{S}(E/\FF)[\pe^{\infty}]$ as the kernel of the localization maps
\[
H^1(\FF, E[\pe^{\infty}]) \to \prod_{v \not\in S} H^1(\FF \otimes_{K} K_v, E[\pe^{\infty}]),
\]
where $v$ runs over finite places of $K$ outside $S$ (note that $v = \ol{\pe}$ is allowed).
Then $\Sel_{S}(E/\FF)[\pe^{\infty}]$ is a discrete $\Z_p[[\Gal(\FF/K)]]$-module.
\end{defn}

Consider $L^{\pe} = L(E[\pe^{\infty}])$, which is a $\Z_p$-extension of $L(E[\pe])$.
Here the extension $L(E[\pe])/K$ is abelian and the degree $[L(E[\pe]): L]$ is prime to $p$.

%

\begin{prop}\label{prop:71}
We have a canonical isomorphism
%
%
\[
\Sel_{S}(E/L^{\pe})[\pe^{\infty}] \simeq \Hom(X_{S}(L^{\pe}), E[\pe^{\infty}]).
\]
\end{prop}

\begin{proof}
We know that every finite place $v \neq \pe$ of $K$ has an infinite residue field extension in $L^{\pe}/K$.
Hence we have an exact sequence
\[
0 \to \Hom(X_{S}(L^{\pe}), \Q_p/\Z_p)
\to H^1(L^{\pe}, \Q_p/\Z_p)
\to \prod_{v \not\in S} H^1(L^{\pe} \otimes_K K_v, \Q_p/\Z_p).
\]
Since the action of the absolute Galois group of $L^{\pe}$ on $E[\pe^{\infty}]$ is trivial,
by twisting, we obtain the assertion.
\end{proof}


\subsection{Exact control theorem}\label{subsec:69}

\begin{prop}\label{prop:72}
We have a canonical isomorphism
%
%
\[
\Sel_{S}(E/L)[\pe^{\infty}] \overset{\sim}{\to} \Sel_{S}(E/L^{\pe})[\pe^{\infty}]^{\Gal(L^{\pe}/L)}.
\]
\end{prop}

\begin{proof}
We follow a standard proof of control theorems (see \cite{Gree99}, for example).
Put $\Gamma = \Gal(L^{\pe}/L)$.
Since we have a natural injective map $\Gamma \hookrightarrow \Aut(E[\pe^{\infty}]) \simeq \Z_p^{\times}$,
the group $\Gamma$ is pro-cyclic.

Consider the commutative diagram with exact rows
\[
\xymatrix{
	0 \ar[r]
	& \Sel_{S}(E/L)[\pe^{\infty}] \ar[r] \ar[d]
	& H^1(L, E[\pe^{\infty}]) \ar[r] \ar[d]
	& \prod_{v \not\in S} H^1(L \otimes_K K_v, E[\pe^{\infty}]) \ar[d]\\
	0 \ar[r]
	& \Sel_{S}(E/L^{\pe})[\pe^{\infty}]^{\Gamma} \ar[r]
	& H^1(L^{\pe}, E[\pe^{\infty}])^{\Gamma} \ar[r]
	& \prod_{v \not\in S} H^1(L^{\pe} \otimes_K K_v, E[\pe^{\infty}])^{\Gamma}\\
}
\]
Thus it is enough to show that both the middle and the right vertical arrows are isomorphic.

First we deal with the middle arrow.
By the inflation-restriction exact sequence, it is enough to show
\[
H^i(\Gamma, E[\pe^{\infty}]) = 0
\]
for $i = 1, 2$.
Since the $p$-cohomological dimension of $\Gamma$ is $1$, we have $H^2(\Gamma, E[\pe^{\infty}]) = 0$.
We use an exact sequence
\[
0 \to H^0(\Gamma, E[\pe^{\infty}]) \to E[\pe^{\infty}] \overset{\gamma - 1}{\to} E[\pe^{\infty}] \to H^1(\Gamma, E[\pe^{\infty}]) \to 0,
\]
where $\gamma$ is a topological generator of $\Gamma$.
Since $H^0(\Gamma, E[\pe^{\infty}]) = E(L)[\pe^{\infty}]$ is finite, this sequence shows that $H^1(\Gamma, E[\pe^{\infty}]) = 0$.

Second we deal with the right arrow.
Let $w$ be a finite place of $L^{\pe}$ which does not lie above $\{\pe\}$.
Then it is enough to show that the restriction map
\[
H^1(L_w, E[\pe^{\infty}]) \to H^1((L^{\pe})_w, E[\pe^{\infty}])^{\Gamma_w}
\]
is isomorphic, where $\Gamma_w = \Gal((L^{\pe})_w/L_w)$.
This can be proved in a similar way to the global case above.
\end{proof}

\subsection{Proof of Theorem \ref{thm:75}}\label{subsec:74}
Recall that $\chi_{E, \pe}$ denotes the character defined by $E[\pe^{\infty}]$.
By Corollary \ref{cor:73} for the extension $L^{\pe}/K$ and Proposition \ref{prop:71}, we obtain
\begin{align}
\Fitt_{\RR^{\pe}} (\Sel_{S}(E/L^{\pe})[\pe^{\infty}]^{\dual})
& = \ttilde{\chi_{E, \pe}} \left(\Fitt_{\RR^{\pe}} (X_{S}(L^{\pe})) \right)\\
& = \ttilde{\chi_{E, \pe}} \left(\iL_{\Sigma, S}(L^{\pe}/K) \Fitt^{[1]}_{\RR^{\pe}} (Z_{\Sigma \setminus S}^0(L^{\pe}/K)) \right).
\end{align}
Combining this with Proposition \ref{prop:72} and the functoriality of the Fitting ideals, we obtain Theorem \ref{thm:75}.

\section*{Acknowledgments}

I would like to express my gratitude to Takeshi Tsuji and to Masato Kurihara for their support during the research.
This research was supported by JSPS KAKENHI Grant Number 17J04650, by JSPS KAKENHI Grant Number 19J00763, and by the Program for Leading Graduate Schools (FMSP) at the University of Tokyo.

{
\bibliographystyle{abbrv}
\bibliography{biblio}
}

\end{document}